 \documentclass[epsf, 12pt, reqno]{amsart}
 \usepackage[colorlinks]{hyperref}
\usepackage{epsf}

\usepackage{amsmath,amssymb}

\newcommand{\C}{{\mathbb C}}
\newcommand{\Z}{{\mathbb Z}}
\newcommand{\PP}{{\mathbb P}}
\newcommand{\F}{{\mathbb F}}
\newtheorem{thm}{Theorem}[section]
\newtheorem{prop}[thm]{Proposition}

\newtheorem{lem}[thm]{Lemma}

\theoremstyle{definition}

\newtheorem{defn}[thm]{Definition}
\newtheorem{clm}[thm]{Claim}
\newtheorem{rmrk}[thm]{Remark}
\newtheorem{nots}[thm]{Notation}
\newtheorem{corl}[thm]{Corollary}

\newcommand{\im}{\operatorname{im}}
\newcommand{\rank}{\operatorname{rank}}

\begin{document}

\title [Characterization of Line Arrangements]{ The Characterization of A Line Arrangement
 whose  Fundamental Group of the Complement is  a Direct Sum of Free
Groups}

\author[M. Eliyahu, E. Liberman,  M. Schaps, M. Teicher]
{Meital Eliyahu, Eran Liberman \\ Malka Schaps, Mina Teicher}

\email{\{eliyahm,liberme1,mschaps,teicher\}@macs.biu.ac.il}
\address{
Department of Mathematics, Bar-Ilan University, Ramat-Gan 52900,
Israel }

 \begin{abstract}  Kwai Man Fan proved 
 that if the intersection lattice of a line arrangement does not
contain a cycle, then the fundamental group of its complement  is
a direct sum of infinite and cyclic free groups. He also
conjectured  that the converse is true as well. The main purpose
of this paper is to prove this conjecture.\end{abstract} \maketitle


\section{Introduction}
An arrangement of lines is a finite collection of $\mathbb{C}$
-linear subspaces of dimension $1$. For such an arrangement
$\Sigma \subseteq {\mathbb{C}}^2$, there is a natural projective
arrangement $\Sigma^*$ of lines in $\mathbb{C}\mathbb{P}^2$
associated to it, obtained by adding to each line its
corresponding point at infinity. The problem of finding
connections between the topology of the differentiable structure
of $\mathbb{C}^2-\Sigma$ and the combinatorial theory of $\Sigma$
is one of the main problems in the theory of line arrangements
(see for example
\cite{CoSu}) . The main motivations for studying
the topology of $\mathbb{C}^2-\Sigma$ are derived from the areas
of hypergeometric functions, singularity theory and algebraic
geometry. Note that if $3$ points are on the same line we do not
consider it as  a cycle.

Given an arrangement $ \Sigma $, we define the graph $ G(\Sigma)$
which lies on the arrangement. Its vertices are the multiple
points (with multiplicity $ \geq 3$) and the edges are the
segments between multiple points on lines which pass through more
than one multiple point. If two lines occur to meet in a simple
point we ignore it (i.e. we do not consider it as a vertex of the
graph).

\medskip
In 1994, Jiang and Yau \cite{JY} defined the concept of a "nice"
arrangement. For $\Sigma$, they define a graph $G(V,E)$: The
vertices are the multiple points of $\Sigma$. $u,v$ are connected
if there exists $l \in \Sigma$ such that: $u,v \in l $. For $v \in
V$ define a subgraph $G_{\Sigma}(v)$: The vertex set is $v$ and
all his neighbors from $\Sigma$ (pay attention this definition
differs from the definition of fan).

$\Sigma$ is \textit{nice} if there is $ V' \subset V$ such that
$E_{\Sigma}(v) \cap  E_{\Sigma}(u) = \emptyset$ for all $u,v \in
V'$, and if we delete the vertex $v$ and the edges of its subgraph
$G_{\Sigma}(v)$ from $G$, for all $v \in V'$ we get a forest (i.e.
a graph without cycles).

 They have proven several properties of those arrangements:
\begin{enumerate}
\item[(a)] If $A_1,A_2$ are nice arrangements and their lattices
are isomorphic, then their complements $\mathbb{C}^2-A_1$ and
$\mathbb{C}^2-A_2$ are diffeomorphic. This property naturally
implies that $\pi_1({\C}^2 -A_1) \cong \pi_1({\C}^2 -A_2)$.

\item[(b)]As a consequence of (a), they showed that the
presentation of the fundamental group of the complement can be
written explicitly and depends only on the lattice of the line
arrangement.

\end{enumerate}
In 2005, Wang and Yau \cite{WY} continued this direction and
proved that the results of Jiang-Yau hold for a much larger family
of arrangements, which they call \textit{simple arrangements}.

 Falk  \cite{Flk} shows  several examples of line
arrangements with the same homotopy types but with different
lattices. Fan \cite{Fa1} proved that up to 6 lines the fundamental
group of a real line arrangement is determined by the lattice
Garber, Teicher and Vishne proved  it for a real line arrangement
with up to 8 lines.

Let $ G, H$ be groups that their abelianizations are free abelian
groups of finite rank.  Choudary, Dimca and Papadima \cite{CDP}
defined a set $\Phi$ of natural group isomorphisms
$\phi:G\rightarrow H $ (which they call \textit{$1$-marking}). In
the case of $G=\pi_1(\mathbb{C}^2-\Sigma)$ (where $\Sigma$ is an
affine arrangement), $\phi$ takes  the topological structure of
$\mathbb{C}^2-\Sigma$ into consideration. They prove that if $A,B$
are line arrangements, and B is "nice" (in the meaning of
Jiang-Yau), then the lattices are isomorphic if and only if there
is an isomorphism $\phi\in\Phi$ where
$\phi:\pi_1(\mathbb{C}^2-A)\widetilde{\rightarrow}
\pi_1(\mathbb{C}^2-B)$.


\vskip 0.7cm

 Fan \cite {Fa2} showed  that if the graph $G(\Sigma)$ is a forest (i.e. a graph
 without cycles),then the fundamental group is a direct sum of free groups.
 He also conjectured that the converse of his theorem is true. In
 \cite{Fa3} Fan proved that if the fundamental group of the
 complement is a direct sum of free groups, then the arrangement
 is composed of parallel lines.

 \vskip0.5cm

  In this paper, we prove his conjecture. Our theorem will state that if the
fundamental group is isomorphic to a direct sum of free groups,
then the graph has no cycles. We would like to emphasize that we
make no restrictions on our isomorphisms.

\medskip

The structure of the paper is as follows. In Section 2, we give
basic definitions related to groups. In section 3 we give basic
definitions related to line arrangements and the fundamental group
of the complement of line arrangements. In Section 4, we define a
function and a special set induced by it. Section 5 deals with
some special properties of fundamental groups of the complement of
line arrangements which are direct sum of free groups. In Section
6, we prove the main result of the paper.

\section{Definitions and notations}
This section presents the needed definitions for the paper.

\subsection{Lower Central Series}
We start by defining the \emph{lower central series of a group $
G$} which will be used throughout the paper.

\begin{defn}{\textbf{ Commutator group and  lower central series} }

Let $G$  be a group . The \textit{commutator group} of $G$ is
\[G'=G_2=[G,G]=\left\langle\{aba^{-1}b^{-1}|a,b\in G\} \right\rangle\]
The subgroup $G'$ is normal in $ G $ with an abelian quotient.
  We can define the \textit{lower central series of $G$} recursively:
\[ G_1=G \]
\[ G_2= [G,G]\]
\[ G_3=[G,G_2] \]
\[\vdots\]
\[G_n=[G,G_{n-1}].\]
Since  $G_{n+1} $ contains the commutators of $G_{n}$, we have $
G_{n+1} \triangleleft G_n$ and the quotient $G_n/G_{n+1}$ is
abelian for  all $n\in \mathbb{N}$.

\end{defn}
To understand these groups, the following identities are needed.
\begin{prop} {\textbf{ Witt-Hall Identities } }\cite{MKS}\label{WHI}
\begin{enumerate}
\item$[a,b][b,a]=e;$ \item$[a,bc]=[a,b][a,c][[c,a],b];$
\item$[ab,c]=[a,[b,c]][b,c][a,c] .$

\end{enumerate}
\end{prop}

  From the second and third identities  we get
\begin{lem}
Let $G$ be a group and let $\{x_1,\ldots ,x_k\}$ be the generators
of $G$.\break Then:
$$G_2/G_3= \langle[x_i,x_j]|
  \hskip 0.3cm i\neq j ,\hskip 0.3cm 1 \leq i,j \leq k\rangle.$$
  \end{lem}

\begin{proof} If $  [xy,z]\in G_2$, then
\[[xy,z]=xyz(xy)^{-1}z^{-1}=xyzy^{-1}x^{-1}z^{-1}=x[y,z]x^{-1}[x,z].\]

Over $G_3$, we know that  $[x,[y,z]]= x[y,z]x^{-1}[y,z]^{-1}
=e\Rightarrow
 [y,z]=x[y,z]x^{-1} .$
 We get
\[[xy,z]=[y,z][x,z] \hskip 0.2cm \mbox{over $G_3.$ }\]
This result means that $G_2/G_3$ is finitely generated by the set
$$\{[x_i,x_j]|i \neq j , 1 \leq i,j\leq k\}$$, where $
{x_1,\ldots,x_k}$ are  the generators of $G$.
\end{proof}

\begin{prop}\label{glob rep}
Let $G=\left\langle x_1,\ldots,x_k \middle| R\right \rangle$,
where the relations $R$ are commutator type relations (i.e. every
relation $r \in R $ can be written as $r=[w_1,w_2],w_1,w_2 \in G$
). Then
$$G_2 /G_3  = \left\langle [ x_i ,x_j ] \middle|
\begin{gathered}
  {[ x _k ,x _l  ] =  {[ x _l ,x _k ]}^{-1}} \hfill, \\
  \text{all  generators  commute} \hfill,\\
  R^*  \hfill \\
\end{gathered}    \right \rangle$$
where $R^*$ is set of the relations $R$  written by means of the
generators of $G_2$, taken modulo $G_3$.
\end{prop}

\medskip
The next Definition  and theorem will give us a better
understanding of $G_2/G_3$ and help us in the future.

\begin{defn} [\cite{Ha}, p.165]Let $G$ be a group generated by $x_1,\ldots,x_r$. We consider
formal words or strings $b_1\cdot b_2 \cdots b_n$ where each $b$
is one of the generators. We also introduce formal commutators
$c_j$ and weights $ \omega(c_j)$ by the rules:
\begin{enumerate}
\item $c_i=x_i, i=1,\ldots,r$ are the commutators of weight 1;
i.e. $\omega(x_i)=1$.
\item If $c_i$ and $c_j$ are commutators,
then $c_k=[c_i,c_j]$ and $\omega(c_k)=\omega(c_i)+\omega(c_j)$.
\end{enumerate}
\end{defn}

\begin{thm}{\bf Basis Theorem}\cite{Ha}
If F is the free group with free generators $ y_1,...,y_r$ and if
in a sequence of basic commutators $c_1,\ldots ,c_t$ are those of
weights  $1,2,\ldots,n$ then  an arbitrary element $f\in F$ has a
unique representation $f=c_1^{e_1}c_2^{e_2}\cdots c_t^{e_t}\: {\bf
mod} \: F_{n+1}$.\\The  basic commutators of weight $n$ form a
basis for the free abelian group $F_n/F_{n+1}$.
\end{thm}

\subsection{Line Arrangements}
\begin{defn} \textbf{Line arrangement.}
\newline
A line arrangement $\Sigma=\{L_1,\ldots, L_s\} \subseteq {\C}^2$
is a union of copies of ${\C}^1$.
\end{defn}
%
%
%
%
%
%
%
%
%
%
%

\begin{rmrk}
Each time we are mentioning a line arrangement $\Sigma
\subseteq{\C}^2$, we assume that there are no parallel lines in
$\Sigma$. This assumption does not restrict us, since we can
project every line arrangement in $\C{\mathbb P}^2$ to $\C^2$ such
that the new arrangement does not contain parallel lines.
\end{rmrk}

\begin{defn}Let $\Sigma$ be a line arrangement.
An intersection point in $\Sigma$ is called \textit{simple} if
there are precisely two lines which meet at that point. Otherwise,
we call it \textit{multiple}.
\end{defn}

\begin{defn}{\textbf{A Cycle}}

A cycle is a non empty ordered set of multiple intersection points
$ \{p_1,\dots,p_k\}$, such that any pair of adjacent points $p_j,
p_{j+1}$ and the points $p_1,p_k$ are connected by lines of the
arrangement. Moreover, if $1 \leq j\leq k-2$, then the line
connecting $p_j$ to $p_{j+1}$ and the line connecting $p_{j+1}$ to
$p_{j+2}$ are different. Also, the line connecting $p_{k-1}$ to
$p_k$ is different from the line connecting $p_k$ to $p_1$.

\end{defn}

Let $\Sigma\subseteq \C \mathbb{P}^2$ be a line arrangement. The
invariant $\beta(\Sigma)$ which is defined in \cite{Fa2}, counts
the number of independent cycles in the graph.

Fan \cite{Fa2} showed that $\beta(\Sigma)=0 $ if and only if $G(\Sigma)$ has no cycles.

In an analogous way for $\Sigma\subseteq {\C}^2$, we define
$\beta(\Sigma):=\beta(\Sigma \cup L_{\infty}) $ where $L_{\infty}$
is the projective line at infinity of ${\C}^2$ and $G(\Sigma):=
G(\Sigma \cup \L_{\infty}) $. Note that the new definitions are
well defined.
\medskip

 One of the most important invariants of a line arrangement is the fundamental group of its
complement, denoted by $\pi_1({\C}^2 - \Sigma)$.

The next lemma presents its computation.

\begin{lem}{\textbf{Constructing the Fundamental Group} }  \label{BFG}(\cite {Arv},\cite{OT} \cite[p.304]{CoSu})
Let  $\Sigma=\{ L_1,\ldots,L_n\} \subseteq {\C}^2 $ be a line
arrangement, that is enumerated as in \cite{CoSu}.

We associate a generator $\Gamma_i$ to each line $L_i$ such that
$$G=\pi_1(\C^2-\Sigma)=\langle\Gamma_1,\dots,\Gamma_n|R \rangle,$$
where $R$ is a set of relations generated as follows.

Every intersection point of lines $L_{i_1},\ldots, L_{i_m}$
creates a set of relations
\[{\Gamma_{i_1}^{x_1}}{\Gamma_{i_2}^{x_2}}\cdots{\Gamma_{i_m}^{x_m}}={\Gamma_{i_m}^{x_m}}
{\Gamma_{i_1}^{x_1}}\cdots{\Gamma_{i_{m-1}}^{x_{m-1}}}={\Gamma_{i_2}^{x_2}}\cdots{\Gamma_{i_m}^{x_m}}{\Gamma_{i_1}^{x_1}}\]
where $x_i\in G$ and $\Gamma_i^{x_i}={x_i}^{-1}\Gamma_i x_i $.

It is easy to see that this set is equivalent to the following
set:

 $$[\Gamma_{i_j}^{x_j},{\Gamma_{i_1}^{x_1}}\cdots{\Gamma_{i_m}^{x_m}}]=e , 1\leq j \leq m.$$
\end{lem}

We get a similar presentation for a real arrangement by using the
Moishezon-Teicher algorithm  \cite {MoTe1} and the van-Kampen
Theorem \cite{VK}.

\begin{nots}
We denote by $\mathcal  {P}$ the set of  intersection points. For
$p \in \mathcal{P}$, we denote by $\Gamma(p)$ the set of generators
attached to the lines passing through the point ${p}$ and by
${\Gamma(p)}^c$ the set of generators attached to the lines not
passing through the point ${p}$.

\end{nots}

\section{decomposition of $G_2/G_3$}


Let $G$ be a group. The abelianization of $G$, denoted by
$\overline{G}$, is $\overline{G}=G/G_2$. If $g \in G$, we denote
$\overline{g}=g\cdot G_2 $.

\begin{rmrk}
Let $\Sigma \in {\C}^2$ be a line arrangement, and let $p \in
\mathcal{P}$.\break Then, $G=\pi_1({\C}^2- \Sigma)=\langle
\Gamma(p), {\Gamma(p)}^c | R\rangle$, $$G/G_2=\overline{G}=Ab(G)=
\langle \overline{\Gamma(p)}, \overline{\Gamma(p)^c
}|R,[x,y]=e,x,y \in \Gamma(p)\cup \Gamma(p)^c \rangle .$$
\end{rmrk}

This Lemma is an immediate implementation of the last section.

\begin{lem}{ \textbf{An implementation for line arrangements.}  }
Let $\Sigma$ be a line arrangement, then the abelian group $G_2
/G_3 $ can be written  as
$$G_2/G_3=\left\langle[\Gamma_i,\Gamma_j] \left|
\begin{gathered}
{[\Gamma _i ,\Gamma _j ] =  [\Gamma _j ,\Gamma _i ]^{-1}}, \hfill
\\ [\Gamma _i ,\Gamma _j][\Gamma _k ,\Gamma _l] =  [\Gamma _k ,\Gamma _l][\Gamma _i ,\Gamma
_j],\\
\prod\limits_{\Gamma_x \in \Gamma(p)}{[\Gamma_x,\Gamma_y],p \in   \mathcal{P},\Gamma_y\in \Gamma(p)}\hfill \\
\end{gathered} \right. \right\rangle.$$
\end{lem}

\begin{proof} A simple implementation of Proposition \ref{glob rep} on the
presentation of $G$ from Lemma \ref{BFG}.
\end{proof}

\begin{rmrk}\label{ds}
We can see that if $\Gamma_1$ and $\Gamma_2$ are associated with
lines meeting in one point and $\Gamma_3$ and $\Gamma_4$ are
associated with lines meeting in a different  point,  there is no
relation combining $[\Gamma_1,\Gamma_2]$ and
$[\Gamma_3,\Gamma_4].$ Therefore,  $$G_2
/G_3=\bigoplus\limits_{p\in   \mathcal{P}}C_p$$ where
$C_p=\left\langle[\Gamma_i,\Gamma_j],\Gamma_i,\Gamma_j\in
\Gamma(p) \left|  \begin{gathered}
{[\Gamma _i ,\Gamma _j ] = [\Gamma _j ,\Gamma _i ]^{-1}} \hfill \\
[\Gamma _i ,\Gamma _j ][\Gamma _k ,\Gamma _l]=[\Gamma _k ,\Gamma
_l][\Gamma _i ,\Gamma _j ],\Gamma _i,\Gamma _j,\Gamma _k,\Gamma _l \in \Gamma(p) \\
\prod\limits_{\Gamma_x \in \Gamma(p)}{[\Gamma_x,\Gamma_y],\Gamma_y\in \Gamma(p)}\hfill \\
\end{gathered} \right. \right\rangle.$
\vskip1cm
 We can see that the generators of the different groups $C_p$ in the direct sum are
 disjoint. Consequently,
let $x\in G_2/G_3$, then $x=\bigoplus_{p \in\mathcal  {P}} {c_p}$,
where $c_p \in C_p$.

 For each $r\in \mathcal{P} $, consider the projection  $$\xi_r:G_2/G_3\rightarrow C_r$$  given by
 $$\xi_r(x)=\xi_r\Big(\bigoplus\limits_{p \in\mathcal  {P}} {c_p}
 \Big)=c_r.$$

If $\Gamma_i \in {\Gamma(r)}^c$, then $ \forall \Gamma_j,
\xi_r([\Gamma_i,\Gamma_j])=\xi_r([\Gamma_j,\Gamma_i])=e$. If
$\Gamma_i, \Gamma_j \in {\Gamma(r)}$, then
$\xi_r([\Gamma_i,\Gamma_j])=[\Gamma_i,\Gamma_j]$.

\end{rmrk}

\section{the stabilizer of an intersection point}
Let $G=\pi_1(\C^2- \Sigma)$.
 Define \[f:G/G_2 \times G/G_2 \longrightarrow
G_2/G_3 \] by
\[f(\overline{a},\overline{b})=[a,b]/G_3. \]

This function is well-defined: If $\overline{a_1}=\overline{a_2}$
and $\overline{b_1}=\overline{b_2}$, then $a_2=a_1x$ where $x\in
G_2$, $b_2=b_1y$ where $y\in G_2$. Then, by Proposition \ref{WHI}:
 $$[a_2,b_2]=[a_1x,b_1y]=[a_1x,b_1][a_1x,y]=[a_1x,b_1]=[a_1,b_1][x,b_1]=[a_1,b_1],$$
 so $f(\overline{a_1},\overline{b_1})=f(\overline{a_2},\overline{b_2} )$.
\vskip1cm

 The following lemma presents some properties of $f$:
\begin{lem} {}\quad \label{additive}
Let $a,b,c \in G/G_2$. Then:
\begin{enumerate}
 \item $f(a \cdot b,c)=f(a,c)\cdot f(b,c).$ \item
$f(a,b \cdot c)=f(a,b)\cdot f(a,c).$

\item $f(a^n,b^m)={f(a,b)}^{nm} \mbox{  for  } m,n \in \Z.$ \item
$f(b,a)=(f(a,b))^{-1}$.

\end{enumerate}
\end{lem}

\begin{proof}
\hskip 0.5cm
\begin{enumerate}
\item Let $A, B,C \in G$ such that $\overline{A}=a$,
$\overline{B}=b$, and $\overline{C}=c.$
 This means that $\overline{AB}=ab$, so by definition,
  $f(ab,c)=[AB,C]/G_3=([A,C][B,C])/G_3$ which by definition is equal
    to $f(a,c)f(b,c)$.

\item The proof is the same as $(1)$.

\item Simple induction on $(1)$ and $(2)$.

\item Let $A,B \in G$ such that $\overline{A}=a$,
$\overline{B}=b$. By definition
$$f(b,a)=[B,A]/G_3=\left({[A,B]/G_3}\right)^{-1}=(f(a,b))^{-1} $$
\end{enumerate}
 \end{proof}

  Now  for any $\overline{x} \in G/G_2$ we define $$S(\overline{x})=\{y \in
G/G_2|f(\overline{y},\overline{x})=e_{G_3}\}.$$ By Lemma
\ref{additive}, $S(\overline{x})$ is a subgroup of $G/G_2$ .
\medskip
 From now on, we talk about a specific intersection point $Q$.
The lines passing through this point are
$\{L_{i_1},\ldots,L_{i_m}\}$.  Define:
$M=\Gamma_{i_1}^{x_1}\Gamma_{i_2}^{x_2}\cdots\Gamma_{i_m}^{x_m}$.
Then, as noted in Lemma \ref{BFG}, the relations induced from the
point Q  can be translated to
$[\Gamma_{i_1}^{x_1},M]=\cdots=[\Gamma_{i_m}^{x_m},M]=e.$

Let $\overline{M}=\overline{\Gamma_{i_1}} \cdot
\overline{\Gamma_{i_2}} \cdots \overline{\Gamma_{i_m}}$.


Since for each $j$,
$\overline{\Gamma_{i_j}}=\overline{\Gamma_{i_j}^{x_j}}$ we get
that

\begin{equation} \label{f(gamma,M)}
f(\overline{\Gamma_{i_j}},\overline{M})=[\Gamma_{i_j}^{x_j},M]/G_3=e_{G_3}.
\end{equation}

\begin{thm}\label{SM}

Let $Q\in \mathcal{P}$ be an intersection point.
 Let $\Gamma(Q)=\{\Gamma_{i_1},\ldots,\Gamma_{i_m} \}$ and
$M={\Gamma_{i_1}^{x_1}} \cdots {\Gamma_{i_m}^{x_m}} $.

 Then
\[
S(\overline{M} ) = \bigg\langle  \overline{\Gamma(Q)}   \cup
\bigg(\bigcap\limits_{\Gamma \in \Gamma(Q)} S(\overline{\Gamma}
)\bigg) \bigg \rangle.
\]
We call $S(\overline{M})$ the {\em stabilizer} of the intersection
point Q.
 \end{thm}

\begin{proof}

We start by proving that $ S(\overline{M} ) \supseteq \bigg\langle
\overline{\Gamma(Q)}   \cup \bigg(\bigcap\limits_{\Gamma \in
\Gamma(Q)} S(\overline{\Gamma} )\bigg) \bigg \rangle. $

 $ \overline{\Gamma(Q)} \subseteq S(\overline{M})$:
We  have already shown that if $\Gamma \in \Gamma(Q)$, then
$f(\overline{\Gamma}, \overline{M} )=\overline{e}$ and hence $
f(\overline{M}, \overline{\Gamma} )=f((\overline{\Gamma},
\overline{M} ))^{-1}=\overline{e}$.
\\
  $\bigcap \limits_{\Gamma \in \Gamma(Q)} S(\overline{\Gamma} )
\subseteq S(\overline{M}) $: If $x\in \bigcap _{\Gamma \in
\Gamma(Q)} S(\overline{\Gamma} )$, then for any $ {\Gamma \in
\Gamma(Q)}, $ \ $f(\overline{\Gamma},x)=e$. This means that also
the product $\prod_{\Gamma \in
\Gamma(Q)}f(\overline{\Gamma},x)=\overline{e}$,
 so by Lemma \ref{additive}
$$
 \overline{e}=\prod_{\Gamma \in
\Gamma(Q)}f(\overline{\Gamma},x)=f \left( \left(\prod\limits_{\Gamma
\in\Gamma(Q)}\overline{\Gamma}\right ),x \right ) =f\left(\overline{M},x\right).$$
 Therefore $x \in S(\overline{M})$.
\vskip1cm

Since $S(\overline{M})$ is a subgroup and we have shown that it
contains the union of $\overline{\Gamma(Q)} $ and
$\bigcap\limits_{\Gamma \in \Gamma(Q)} S(\overline{\Gamma} ), $ it
clearly contains the subgroup generated by the union.
\medskip

To complete the proof, we prove  the opposite inclusion:
\[
S(\overline{M} ) \subseteq \bigg\langle \overline{\Gamma(Q)} \cup
 \bigg(\bigcap\limits_ {\Gamma \in\Gamma(Q)} S(\overline{\Gamma} ) \bigg)
\ \bigg \rangle.
\]
Let $\overline{x} \in S(\overline{M} ) \subseteq G/G_2= \langle
\overline{\Gamma(Q)},\overline{{\Gamma(Q)}^c}|\overline{R}
\rangle$. Then $\overline{x}$ can be written as
$\overline{x}=\overline{z}\cdot \overline{y}$ where $
\overline{z}\in \langle\overline{\Gamma(Q)}\rangle ,
\overline{y}\in \langle {\Gamma(Q)}^c\rangle $. We will prove $
\overline{y}\in \bigcap\limits_ {\Gamma \in\Gamma(Q)}
S(\overline{\Gamma})$.

 Since  $ \overline{z^{-1}} \in
\langle \overline{\Gamma(Q)}\rangle  \subseteq S(\overline{M})$,
so $\overline{y}= \overline{z^{-1}} \cdot \overline{z} \cdot
\overline{y}=\overline{z^{-1}}\overline{x} $. Hence we get
$\overline{y} \in S(\overline{M}) $.

 Let $\l_x $ be a line passing through $Q$  and $\Gamma_x$  be the generator
 associated with it. Recall that $ \overline{y}\in S(\overline{\Gamma})$ if $ f(\overline{\Gamma},
\overline{ y})=\overline{e}$.
  We need to prove that $f(\overline{\Gamma_x},\overline{y})=\overline{e}$ in $G_2/G_3$. From Lemma  \ref{ds}
  we know that $G_2/G_3$ is a direct sum of groups
  $$G_2 /G_3=\bigoplus\limits_{p\in \mathcal{P}}C_p,$$
 so it remains to prove that $ f([\overline{\Gamma_x},\overline{y}])=\overline{e}$ in $C_p$, for all  $p\in \mathcal{P}$.
 Let $p\in \mathcal{P}$. We have to show that  the projection of the coset  $[\overline{\Gamma_x},\overline{y}]/G_3$
  on $C_p$ is trivial.

We separate our proof into three cases.

  {\textbf{Case $1$:}}
 $p=Q .$

 Since $\overline{y}\in  \langle\overline{\Gamma(Q)^c}\rangle= \langle\overline{\Gamma(p)^c}\rangle$
 then the projection is trivial by the definition of $C_p$.

 {\textbf{Case $2$:}}  $\l_x$ does not pass through $p$.

 In this case, $\overline{\Gamma_x} \in \langle \overline{\Gamma(p)^c}\rangle $ and therefore the projection is
 trivial, by the definition of $C_p$.

  {\textbf{Case $3$:}}
 $p \neq Q$  and $\l_x$ does pass through $p$.

 By definition of S, $f(\prod_ { \Gamma \in \Gamma(Q)} \overline{\Gamma },\overline{y})=\overline{e}.$
 This means by the properties of $f$ (Lemma \ref{additive} ) that
 $\prod_{\Gamma \in \Gamma(Q)}f(\overline{\Gamma},\overline{y})=\overline{e}$.
 Since $\l_x$ passes through $p$, all the other lines passing through $Q$ do not cross $p.$
  So if we project any $f(\overline{\Gamma_z},\overline{y})$ where $\Gamma_z$ is any other generator
   from $\Gamma(Q),$
 we get the identity. So in $C_p$,\ $$\overline{e}=\prod_{\Gamma \in \Gamma(Q)}f(\overline{\Gamma},\overline{y})=
 f(\overline{\Gamma_x},\overline{y}),$$ and thus $f(\overline{\Gamma_x},\overline{y})=\overline{e}.$

 \end{proof}

\section {fundamental groups which are semidirect product }
\begin{thm}\label{ds pro}
Let $G=\mathop {\left(\bigoplus _{i = 1}^n A_i \right) }\oplus \mathbb{Z}^l$
where $A_i$ are free groups and  let
$\psi:G\longrightarrow G$  be a projection  where
$\im(\psi)=\langle y_1,\ldots,y_k\rangle \cong{\mathbb{F}}_k$. Then, there
exists $r,1\leq r \leq n$,  such that $\forall g\in G, {\rm pr}_r(g)=e
$ implies that $ \psi (g)=e$ and $\psi=\psi\circ {\rm pr}_r$, where
${\rm pr}_r:G\rightarrow G$
 is the projection onto the subgroup of G naturally isomorphic to $A_r$.

\end{thm}

\begin{proof} Let us denote the  ${\Z}^l$  component as $A_0$.
Since $\psi$ is a projection, $\psi(y_1)=y_1$ and $\psi(y_2)=y_2$.
Then $\psi([y_1,y_2])=[y_1,y_2]\neq e.$ Note: $$[y_1 ,y_2 ] =
\mathop \oplus \limits_{i = 1}^n {\rm pr}_i ([y_1 ,y_2 ]) + {\rm pr}_0 ([y_1
,y_2 ]).$$

Since $A_0$ is abelian, $ {{\rm pr}_0 ([y_1 ,y_2 ])}=e$, and
therefore
$$[y_1 ,y_2 ] = \mathop \oplus \limits_{i = 1}^n {\rm pr}_i ([y_1 ,y_2
]).$$

 Since $ \psi $ is a homomorphism: $$\psi([y_1,y_2])=\psi
({\rm pr}_1([y_1,y_2])\oplus \dots \oplus {\rm pr}_n
([y_1,y_2]))=\psi ({\rm pr}_1([y_1,y_2])) \oplus \cdots \oplus
\psi({\rm pr}_n ([y_1,y_2]))\neq e.$$

 This means that $ \exists r,1 \leq r \leq n $,
such that $\psi({\rm pr}_r [y_1,y_2])\neq e $, so if we denote $a:=
{\rm pr}_r(y_1)$ and $b:={\rm pr}_r(y_2)$ we get $[\psi(a),\psi(b)]\neq e$.

 Since $a,b \in A_r$, if $x\in A_j$ where $j \neq r$, then
$[x,a]=[x,b]=e$.
 We get that
$$[\psi(x),\psi(a)]=[\psi(x),\psi(b)]=e.$$

 So $\psi(x)$ commutes with two noncommutative elements in a free group and hence:
  $ \psi(x)=e ,\forall x\in A_j,$ $j\neq r .$
This means that if $g \in G$ and ${\rm pr}_r(g)=e$, then $\psi(g)=e$.

For each $g \in G$, $g=v_1 \oplus\cdots \oplus v_n \oplus w, v_j\in
A_j, w\in {\Z}^l, $  we have that \break  $\psi(g)=\psi (v_1 \oplus \cdots \oplus
v_n\oplus w)=\psi(v_1)\oplus \cdots \oplus \psi(v_n)\oplus
\psi(w)=\psi(v_r)=\psi({\rm pr}_r(g))$. Therefore $\psi=\psi \circ {\rm pr}_r$.

\end{proof}

For proving the next result, we use the following theorem from
\cite{LS}:

\begin{thm}\label{basis}
Let $\phi$ be an epimorphism from a finitely generated free group
$F$ to a  free group $G$. Then F has a basis $Z=Z_1\bigcup Z_2$
such that  $\phi$ maps $ \langle Z_1 \rangle$ isomorphically onto
$G$ and maps $\langle Z_2 \rangle$ to $e$.
\end{thm}

\begin{corl}\label{genAr}
With the same assumptions of Theorem \ref{ds pro} where
$\im(\psi)=\langle y_1,\ldots,y_k\rangle \cong \mathbb{F}_k$, for
the same $ r$ in Theorem  \ref{ds pro} we can find elements \break
$\{z_1,\ldots, z_m|m\geq 0\} \subseteq A_r$ such that $$\{{\rm pr}_r(y_1),
\ldots,{\rm pr}_r(y_k) ,z_1,\ldots, z_m|m\geq 0\}$$ are generators of
$A_r$ and  $\psi(z_i)=e$.

\begin{proof}

From Theorem \ref{ds pro}  and the fact that $\psi$  is
projection ($\psi \circ \psi=\psi$), we get that
${\rm pr}_r\circ\psi\circ {\rm pr}_r\circ\psi={\rm pr}_r\circ\psi$. Therefore,
${\rm pr}_r\circ \psi$ is a projection.

Now since $({\rm pr}_r\circ\psi) (A_r)\subseteq A_r,$
\[ ({\rm pr}_r\circ\psi)|_{A_r}\circ({\rm pr}_r\circ\psi)|_{A_r}=(({\rm pr}_r\circ\psi)
 \circ ({\rm pr}_r\circ\psi))|_{A_r}={({\rm pr}_r\circ\psi)|_{A_r}}.\]
Given:
 \begin{equation} \im(\psi)=\langle y_1,\ldots,y_k \rangle\quad \text{is a free
 group.}\tag{$\ast$}\end{equation}
Now we claim that: $\{ {\rm pr}_r(y_1),\ldots,{\rm pr}_r(y_k)\}$
are the generators of free group ${\F}_k$. Otherwise, there exists
a nontrivial word $w=W(y_1,\ldots,y_k)$ so that ${\rm
pr_r}(w)=W({\rm pr}_r(y_1),\ldots,{\rm pr}_r(y_k))=e$. Therefore:
$$(\psi\circ {\rm pr}_r)(w)=\psi({\rm pr}_r(w))=\psi(e)=e. $$ By
Theorem  \ref{ds pro}: $\psi=\psi\circ {\rm pr}_r$, therefore
$\psi(w)=(\psi\circ {\rm pr}_r)(w)=e$ which means that $w=e$, a
contradiction to ($\ast$).

In conclusion, we get that  $({\rm pr}_r\circ\psi)|_{A_r}:A_r\rightarrow
A_r$ is a projection, such that \break $  \im({\rm pr}_r\circ\psi)=\langle
{\rm pr}_r(y_1),\ldots,{\rm pr}_r(y_k) \rangle \cong {\F}_k$
(${\rm pr}_r(y_1),\ldots,{\rm pr}_r(y_k)$ are the generators of $\F_k$).

By Theorem \ref{basis}, $A_i$ has a basis $B=B_1\bigcup B_2$ such
that ${\rm im}({\rm pr}_r\circ\psi)= \langle B_1\rangle $ and
$({\rm pr}_r\circ\psi)(B_2)=e$. Since ${\rm pr}_r\circ\psi$ is a
projection, we can assume that \break $B_1= \langle {\rm
pr}_r(y_1),\ldots,{\rm pr}_r(y_k)\rangle$. Let $ B_2=
\{z_1,\ldots,z_m |{\rm pr}_r\circ \psi(z_i)=e\}$, which means that
$ (\psi \circ {\rm pr}_r\circ \psi)(z_i)=e$ . As we proved
earlier, $ \psi \circ {\rm pr}_r =\psi$, so $( \psi \circ
\psi)(z_i)=e$. But $ \psi \circ \psi=\psi$, therefore $
\psi(z_i)=e$.

 This implies that $A_r$ has generators \[\{ {\rm pr}_r(y_1),\ldots, {\rm pr}_r(y_k), z_1,\ldots, z_m\} \]
which satisfy the needed properties.

\end{proof}
\end{corl}
Without loss of generality, we temporarily change to a local
numeration.

Let $G=\pi_1({\C}^2 -\Sigma)=\langle \Gamma_1,\ldots,
\Gamma_n,\Gamma_{n+1},\ldots, \Gamma_l |R \rangle$, where
$\Gamma(A)=\{\Gamma_1,\ldots ,\Gamma_n \}$ are the generators of
the lines which participate in a specific multiple intersection
point called A.

\begin{defn}\cite{Jo}
Let $G$ be a group with a given presentation: $G=\langle
X|R\rangle $. A \textit{Tietze transformation}  $T_i \quad( 1\leq i
\leq4 )$ is a transformation of $\langle X|R\rangle $ to a
presentation $G=\langle X'|R'\rangle $ of one of the following
types:

\begin{enumerate}
\item[(T$_1$)]   If $r$ is a word in $\langle X \rangle$ and $r=e$ is a relation
that holds in $G$, then let $X'=X, R'=R\cup\{r\}. $

\item[(T$_2$)]   If $r\in R$ is such that the relation $r=e$ holds
in the group $ \langle  X|R-{r}\rangle $, then let $X'=X,
R'=R\setminus \{r\}$.

\item[(T$_3$)]   If $w$ is a word in $\langle X \rangle$ and $z \notin X$, put
$X'=X\cup \{z\},R'=R\cup \{wz^{-1}\}$.

 \item[(T$_4$)]    If $z\in X$  and $w$  is a word in the other elements of X
 such that
$wz^{-1}\in R$, then substitute $w$ for $z$ in every other element
of $R$ to get $\widetilde{R}$ and take $X'=X - \{z\}, R'=
\widetilde{R}$.
\end{enumerate}
\end{defn}


Before introducing the notion of a {\it braid monodromy}
\cite{MoTe1}, we have to make some constructions. From now, we
will work in $\C^2$. Let $E$ (resp. $D$) be a closed disk on
$x$-axis (resp. $y$-axis), and let $C$ be a part of an algebraic
curve in $\C^2$ located in $E \times D$. Let $\pi _1 : E \times D
\to E$ and $\pi _2 : E \times D \to D$ be the canonical
projections, and let $\pi = \pi _1 |_C: C \to E$. Assume $\pi$ is
a proper map, and $\deg \pi = n$. Let $ N = \{ x \in E \ |\ \# \pi
^ {-1}(x) < n \} $, and assume $ N \cap \partial E= \emptyset $.
Now choose $ M \in \partial E$ and let $K = K(M) = \pi ^ {-1}
(M)$. By the assumption that $\deg \pi = n \ \ (\Rightarrow \#
K=n)$,  we can write: $K = \{ a_1,a_2, \dots ,a_n \} $. Under
these constructions, from each loop in $E-N$, we can define a
braid in $B_n [M \times D, K]$ in the following way:
\begin{itemize}
\item[(1)] Since $ \deg \pi = n$, we can lift any loop in $E-N$ with
a base point $M$ to a system of $n$ paths in $ (E-N) \times D $ which
start and finish at $ \{ a_1,a_2, \cdots ,a_n \} $.
\item[(2)] Project this system into $D$ (by $\pi _2$), to get $n$ paths in $D$
which start and end at the image of $K$ in $D$ (under $\pi_2$). These paths
actually form a motion.
\item[(3)] Induce a braid from this motion, see \cite{Gthesis}.
\end{itemize}

\medskip

To conclude, we can match a braid to each loop. Therefore, we get
a map $ \alpha : \pi _1(E-N,M) \to B_n [M \times D, K]$, which is
also a group homomorphism. This map is called the {\bf braid
monodromy of $C$ with respect to $E \times D,\pi _1,M$} (see
\cite{MoTe1}, \cite{Gthesis}).

\vskip0.5cm

 The following remark demonstrates the necessity of
the condition that there are no parallel lines in the affine plane
:

\begin{rmrk} \label{center}
 Let $ \Sigma$ be a line arrangement with no parallel line in the affine plane.
 Let $\Gamma_1,\ldots,\Gamma_l $ be the generators of
$\pi_1({\C}^2-\Sigma)$. Then $$\Gamma_1\cdots\Gamma_l \in
Z(\pi_1({\C}^2-\Sigma)).$$
\begin{proof}
Let $\alpha:\pi_1({\C}-\Sigma)\rightarrow B_n$ be the braid
monodromy of $ \Sigma$. By Remark 4.7 in \cite{CoSu}, the closed
braid determined by the product $\alpha(\Gamma_1)\cdots
\alpha(\Gamma_l)$ is actually a link of the curve at infinity.
In a generic curve, we have :
$\alpha(\Gamma_1)\cdots \alpha(\Gamma_l)=\Delta^2$, where $\Delta$
is the braid which rotates by $180^\circ$  counterclockwise all
the strands together. One of the presentations of the fundamental
group is: $\langle \Gamma_1,\ldots, \Gamma_l | \alpha
(s_i)(\Gamma_j)=\Gamma_j  \forall i,\forall j \rangle$, where
$s_i$ is the loop created in the $x$-axis as a result of the
projection of the line arrangement on the plane \cite{VK}.
Therefore, $\alpha(s_i)$ is a braid acting on the generators of
the fundamental group. As a result $\alpha(s_1 \cdots
s_n)(\Gamma_j)=\Gamma_j ,\forall j$, which means $
\Delta^2(\Gamma_j)=\Gamma_j ,\forall j$. It is known that $
\Delta^2(x)=x^{\Gamma_1\cdots\Gamma_l }$ and thus
${\Gamma_j}^{\Gamma_1\cdots\Gamma_l}=\Gamma_j$. Hence,
$\Gamma_1\cdots\Gamma_l \in Z(\pi_1({\C}^2-\Sigma))$.
\end{proof}
\end{rmrk}
\medskip

Let $G=\pi_1({\C}^2 -\Sigma)=\langle \Gamma_1,\ldots
\Gamma_n,\Gamma_{n+1},\ldots, \Gamma_l |R \rangle$, where
$\Gamma(A)=\{\Gamma_1,\ldots ,\Gamma_n \}$ are the generators of
the lines which pass through in a specific multiple intersection
point called A.

Let us recall that the lines passing through $Q$ are
$\{L_{i_1},\ldots,L_{i_m}\}$. Recall also $M=\Gamma_{i_1}^{x_1}\Gamma_{i_2}^{x_2}\cdots\Gamma_{i_m}^{x_m}$.

\begin{thm}\label{semidirect}With the same assumptions of Theorem \ref{ds pro} and Remark \ref{center},
  let $H:=\langle \Gamma_{i_1},\ldots,\Gamma_{i_{m-1}}\rangle $.
  Let $N$ be the normal closure of
 $\{{{\Gamma(Q)}^c},M\}$. Let $f_1:H \rightarrow G$ be the natural embedding and
 $f_2:G \rightarrow G/N$ the natural homomorphism.
 Then:
 \begin{enumerate}
 \item $G/N \cong {\mathbb{F}}_{m-1}$ which is generated by $f_2(\Gamma_{i_j}),$ $ 1\leq j \leq
 m-1$;
 \item $H \cong {\mathbb{F}}_{m-1}$;
 \item $G = N \rtimes {\mathbb{F}}_{m-1}$;
 \item There exists a  projection $h:G\rightarrow G \quad(h^2=h)$, such that $\im(h)=H$ and
$ \ker(h)=N$.

 \end{enumerate}
\end{thm}

\begin{proof}

\begin{enumerate}
\item [(1)]We can present $G/N=\langle\Gamma(Q) ,
\Gamma(Q)^c|R,M,\Gamma(Q)^c\rangle $.
     Denote $\Gamma(Q) ^c=\{z_1,\dots,z_k\}$. By  applying iteratively  Tietze's
    transformation  ($T _4$)  for every $z_i\in\Gamma(Q) ^c$, we
obtain
    $G/N=\langle\Gamma(Q)|\widehat{R}, \widehat{M } \rangle $, where
    $\widehat{R},\widehat{M}$
    are  obtained from $R, M$ by substituting  $e$ for $z_i,$ $1\leq i\leq k$,
     in every other
    element of $R$ and $M$, respectively.

    Let $  v \in \mathcal{P} $ with $v \neq Q$ be an intersection point, and let $r $ be the
    relations induced by $v$.

    Now we can get the following:
\begin{enumerate}
\item[(a)] If there is a line $L$ that goes through both points
$v$ and $Q$  and the generator attached  to $L$ is $\Gamma_{i_j} $
for some $j$, $1 \leq j \leq m$, then the relations $r$ become
$[{\Gamma_{i_j}^x},1,\ldots,1]$ for some $x$. These relations are
trivial.

 \item[(b)] If the points $v$ and $Q$ do not share any line,
then the relations $r$ become $[1,\ldots,1]$, which are also
trivial.
\end{enumerate}
Let us denote by $\widetilde{W}$ the word $W$ after rewriting. Due
to the above observations, using the presentation $ G/N=\langle
\Gamma(Q)|\widetilde{R}, \widetilde{M} \rangle $, then
$$G/N=\langle\Gamma_{i_1},\ldots,\Gamma_{i_m}|[{\Gamma_{i_1}^{\widetilde{x_{i_1}}}},
\ldots{\Gamma_{i_m}^{\widetilde{x_{i_m}}}}],\widetilde{M}\rangle.$$
This is equal to
$$\langle\Gamma_{i_1},\ldots,\Gamma_{i_m}|[{\Gamma_{i_1}^{\widetilde{x_{i_1}}}},\widetilde{M}],\ldots,
[{\Gamma_{i_m}^{\widetilde{x_{i_m}}}},\widetilde{M}],\widetilde{M}\rangle={\mathbb{F}}_{m-1}.$$

 \item[(2)]Assume otherwise. Then there exists a non-trivial word
$w(\Gamma_{i_1},\ldots,\Gamma_{i_{m-1}})= e$. So applying $f_2,$
we get a  non-trivial word  $w(f_2(\Gamma_{i_1}),\ldots,
f_2(\Gamma_{i_{m-1}}))$ which is not possible by (1).

\item[(3)]From the first paragraph, we get that $f_2\circ f_1$ is a
surjective function from $\mathbb{F}_{m-1}$ to $\mathbb{F}_{m-1}$
and therefore it is an isomorphism.

\item[(4)]Derived directly from (2) and (3).
\end{enumerate}
\end{proof}

\medskip

By the last theorem: Let $Q\in\mathcal{P}$, $\Gamma(Q)=\{\Gamma_{i_1},\ldots,\Gamma_{i_m}\}$.
Then there is a projection $h:G\rightarrow G$ such that
${\rm im}(h)=\langle\Gamma_{i_1},\ldots,\Gamma_{i_m}\rangle$ and  ${\rm ker}(
h)$ is the normal closure of $\langle \Gamma(Q)^c,M\rangle$.

By Theorem \ref{ds pro} and Corollary \ref{genAr} we get that $
\exists r,$ $ 1\leq r\leq n,$\ such that \break
$A_r=\{w_1,\ldots,w_{n-1},z_1,\ldots,z_k|k \geq 0 \}$ and $
h=h\circ {\rm pr}_r$, ${\rm pr}_r \circ h$ is a projection,
$\Gamma_i = h(w_i),h(z_i)=e,b_i=\Gamma_i {w_i}^{-1}, h(b_i)=e $
(note that $h$ has the role of $\psi$).
\medskip

\begin{clm}
 $\overline{A_r}={\rm pr}_{\overline{r}}(S(\overline{M}))$.
\end{clm}

\begin{proof}
We start by proving that ${\rm pr}_r(M)=e$: for all $x,y$ and
$j_1\neq j_2$. Since $\Gamma_{i_{j_1}}$ and  $\Gamma_{i_{j_2}}$
are different generators of the  free group ${\rm im}(h)$:
$h([\Gamma^x_{i_{j_1}}\Gamma^y_{i_{j_2}}])\neq e$. We know that
$h=h\circ {\rm pr}_r$, so we get $(h\circ {\rm
pr}_r)([\Gamma^x_{i_{j_1}}\Gamma^y_{i_{j_2}}]) = h (
[\Gamma^x_{i_{j_1}}\Gamma^y_{i_{j_2}}])\neq e$, therefore $ {\rm
pr}_r([\Gamma^x_{i_{j_1}}\Gamma^y_{i_{j_2}}]) \neq e$. Hence, for
$x=x_{j_1},y=x_{j_2}$, $ {\rm pr}_r \left
([\Gamma^{x_{j_1}}_{i_{j_1}},\Gamma^{x_{j_2}}_{i_{j_2}}]\right
)\neq e$. We know that $[\Gamma^{x_{j_1}}_{i_{j_1}},M]=e$, so the
elements ${\rm pr}_r(\Gamma^{x_j}_{i_j}),\quad j=1,2,$ commute
with $M$ but do not
commute with each  other in the free group $A_r$, therefore ${\rm pr}_r(M)=e$. \\
 The next step is to prove that $\overline{A_r}\subseteq S(\overline{M})$. Let $a\in
\overline{A_r}$.  Then there exists $a^*\in A_r$ such that
$\overline{a^*}=a$.  Since ${\rm pr}_r(M)=e$ and $a^*\in A_r$, by the
properties of direct sum $[a^*,M]=e.$  Therefore
 $[a^*,M]/G_3=e_{G_3}$. By definition $f(a,\overline{M} )=e_{G_3}$ which implies that $a\in S(\overline{M})$.

Hence, we have $\overline{A_r}\subseteq {\rm pr}_{\overline{r}}(S(\overline{M}))$.
The opposite inclusion is trivial by definition, and thus $\overline{A_r}={\rm pr}_{\overline{r}}(S(\overline{M}))$.
\end{proof}

\medskip



 By Theorem \ref{SM},  $S(\overline{M})= \left\langle {\left\{
{\overline{\Gamma _{i_1}} ,\dots,\overline{\Gamma _{i_m} }}
\right\} \cup \left\{ {\bigcap _{k = 1}^m {S(\overline{\Gamma _{i_k}} )} }
\right\}} \right\rangle $, therefore:

\begin{eqnarray*}\overline{A_r}&=&{\rm pr}_{\overline{r}}\left(\left\langle {\left\{ {\Gamma
_{i_1} ,\dots,\Gamma _{i_m} } \right\} \cup  {\bigcap _{k =
1}^m {S(\Gamma _{i_k} )} } } \right\rangle \right)
\\&=&\left\langle {\left\{ {\rm pr}_{\overline{r}}({\Gamma
_{i_1}}),\dots,{\rm pr}_{\overline{r}}({\Gamma _{i_m}})
 \right\} \cup  {\rm pr}_{\overline{r}}
 \left(  {\bigcap _{k = 1}^m {S(\overline{\Gamma _{i_k}} )} }\right)
 }\right\rangle.\end{eqnarray*}

\medskip


We claim that the right part of this generating set is trivial:

\begin{clm}
 $$ {\rm pr}_{\overline{r}}\Big(   {\bigcap\limits_{k = 1}^m {S(\Gamma
_{i_k})} } \Big)=\{e\}.$$

\end{clm}

\begin{proof}
It is known that for any two sets $ A,B$ and a function $F$ that
$F(A \cap B)\subseteq F(A) \cap F(B)$. Therefore ${\rm
pr}_{\overline{{r}}}\left(\bigcap\limits_{k=1}^{m}S(\Gamma_{i_k})\right)\subseteq
\bigcap\limits_{k=1}^{m}[{\rm
pr}_{\overline{{r}}}(S(\Gamma_{i_k}))]$.
\\As we mentioned $A_r=\langle
{\rm pr}_r(\Gamma_1),\ldots ,{\rm pr}_r(\Gamma_{m}),z_1,\ldots,z_q|q\geq
0\rangle $.

\medskip
In other words: $ A_r=\left<\beta_1,\ldots,\beta_{m+q}\right> $
where

   $$\beta _k  = \left\{ \begin{gathered}
  {\rm pr}_r (\Gamma_{i_k} ){\text{    }} \quad k \leqslant m \hfill \\
  z_{k-m} \qquad {\text{            }}m+1 \leqslant k \leqslant m+q {\text{  }} \hfill \\
\end{gathered}  \right.$$

%

Let $\overline{g}\in {\rm
pr}_{\overline{{r}}}(S(\Gamma_{i_k}))\subseteq \overline{A_r}$,
therefore $\overline{g}=\overline{{\beta_1}^{t_1}}\cdots
\overline{{\beta_{q+m}}^{t_{m+q}}}$. Without loss of generality,
we can choose that $g=\beta_1^{t_1}\cdots \beta_{m+q}^{t_{m+q}}$.
By Theorem \ref{ds pro} we know that $$G=A_1\oplus\cdots\oplus
A_r\oplus\cdots\oplus A_n\oplus \mathbb{Z}^\ell.$$

\medskip
 We define $D=A_1\oplus\cdots\oplus \widehat{A_r}\oplus\cdots\oplus
A_n\oplus \mathbb{Z}^\ell$, then $G=A_r\bigoplus D$.


Hence, there exists $g_2\in D$ such that
$\overline{g}\oplus\overline{g_2}\in S(\overline{\Gamma_{i_k}})$.

\medskip


Let $\varphi :G\rightarrow D $ be the natural projection, and define $\eta_k:=\varphi (\Gamma_{i_k})$.
Hence, it implies that $\Gamma_{i_k}={{\rm pr}_r}(\Gamma_{i_k})\oplus\eta_k$.
\medskip

\begin{eqnarray*}e_{G_3}&=&f(\overline{\Gamma_{i_k}},\overline{g\oplus
g_2})=f(\overline{{{\rm pr}_r}(\Gamma_{i_k})\oplus\eta_k},\overline{g \oplus g_2})\\
&=&f(\overline{{{\rm pr}_r}(\Gamma_{i_k})},\overline{g})f(\overline{{{\rm pr}_r}(\Gamma_{i_k})},\overline{g_2})f(\overline{\eta_k},
\overline{g})f(\overline{\eta_k},\overline{g_2}).\end{eqnarray*}


By the definition of $f$, this implies that
$$[\overline{{{\rm pr}_r}(\Gamma_{i_k})},\overline{g}][\overline{{{\rm pr}_r}(\Gamma_{i_k})},\overline{g_2}][\overline{\eta_k},
\overline{g}][\overline{\eta_k},\overline{g_2}]\in G_3.$$

Since ${\rm pr}_r({\Gamma}_{i_k}),g \in A_r$, and $ g_2,\eta_k \in D$,
then $[{\rm pr}_r({\Gamma}_{i_k}),g_2]=[\eta_k ,g]=e_{G_3} $. Therefore
$[{\rm pr}_r({\Gamma}_{i_k}),g][\eta_k ,g_2] \in G_3=(A_r)_3 \oplus D_3
$, which means $ [{\rm pr}_r({\Gamma}_{i_k}),g]\in(A_r)_3 ,[\eta_k
,g_2]\in D_3 $.

The fact that $ (A_r)_2/(A_r)_3= \left\langle
[\beta_i,\beta_j]|1\leq i < j \leq m+q \right\rangle$ is derived
directly from the definition of $ A_r$. Hence, $({A_r})_3\ni
[{\rm pr}_r\left({\Gamma}_{i_k}\right),g]=[\beta_k,g]=[\beta_k,\beta_1^{t_1}\cdots
\beta_{m+q}^{t_{m+q}}]=[\beta_k,\beta_1]^{t_1}\cdots
[\beta_k,\beta_k]^{t_k}\cdots
[\beta_k,\beta_{m+q}]^{t_{m+q}}$.

Note that $[\beta_k,\beta_k]=e$. By \cite{Ha},
$\{[\beta_k,\beta_j]|j\neq k\} $ are generators (or inverses) of
$(A_r)_2/(A_r)_3$ which is a free  abelian group. Since every
generator appears at most once, we get $ t_j=0 \quad $ for all $
j\neq k$ .
\\ As a conclusion, $\overline{g}=\overline{{\beta_k}^{t_k}} \in \langle \overline{\beta_k} \rangle=
\langle \overline{{\rm pr}_r \left({\Gamma}_{i_k}\right)} \rangle$. So
 $\overline{g }\in\langle\overline{ {\rm pr}_{\overline{r}} \left({\Gamma}_{i_k}\right)}\rangle
 $.
\\To summarize, we have shown that if  $ \overline{g }\in \langle
{\rm pr}_{\overline{r} }\left( S({\Gamma}_{i_k})\right)\rangle$
then $\overline{g }\in\langle {\rm pr}_{\overline{r} }\left(
{\Gamma}_{i_k}\right)\rangle$. Therefore, $\langle {\rm
pr}_{\overline{r} }\left( S({\Gamma}_{i_k})\right)\rangle
\subseteq \langle {\rm pr}_{\overline{r}}
\left({\Gamma}_{i_k}\right)\rangle$.
\\
As a conclusion,  $ \bigcap \limits_{k = 1}^m \langle {\rm
pr}_{\overline{r} }\left( S({\Gamma}_{i_k})\right)\rangle
\subseteq \bigcap \limits_{k = 1}^m  \langle {\rm
pr}_{\overline{r}} \left({\Gamma}_{i_k}\right)\rangle={e}$, and
hence $ {\rm pr}_{\overline{r}}\Big(   {\bigcap\limits_{k = 1}^m
{S(\Gamma_{i_k})} } \Big)=\{e\}$ as claimed.
\end{proof}

\bigskip

Since $\Gamma_{i_1}\cdots\Gamma_{i_m}=e$,  $\overline{A_r}= \langle
\overline{\Gamma_{i_1}},\ldots,\overline{\Gamma_{{i_{m-1}}}}\rangle$.
As a result, $\overline{A_r}={\rm pr}_{\overline{r}}\langle { {\Gamma
_{i_1} ,\dots,\Gamma _{i_{m-1}} } \rangle} $. Consequently,
$\rank(A_r)\leq m-1 $.

We know that $A_r=\langle w_1,\ldots,w_{m-1},z_1,\ldots,z_t|t\geq 0\rangle$.
Combining these two facts together and by Corollary \ref{basis}, we get
$A_r= \langle w_1,\ldots,w_{m-1} \rangle$, where \break $y_i=({\rm pr}_r\circ
h)(\Gamma_{i_j}^{x_j})$.


 From Theorem \ref{ds pro}, $h\circ {\rm pr}_r =h$. Since $h$ is a
 projection  $h^2=h$.
 Hence, we get : $${\rm pr}_r \circ h={\rm pr}_r \circ h \circ {\rm pr}_r \circ h=
 {\rm pr}_r \circ h \circ h={\rm pr}_r \circ h ,$$
which means that ${\rm pr}_r \circ h$ is a projection too. Recall
that $$A_r=\langle ({\rm pr}_r \circ h)(\Gamma_{i_1}),\ldots,({\rm
pr}_r \circ h)(\Gamma_{i_{m-1}}) \rangle.$$ Therefore $({\rm pr}_r
\circ h)(w_j)= ({\rm pr}_r \circ h )\circ ({\rm pr}_r \circ
h)(\Gamma_{i_j})=
 ({\rm pr}_r \circ h)(\Gamma_{i_j})=w_j$.
Hence we get that $ ({\rm pr}_r \circ h)(A_r)=A_r$.

Now if $l_i$ is a line that does not pass through the point $Q$
and $\Gamma_i$ is the generator of $l_i$, then $ h(\Gamma_i)=e$
and therefore $({\rm pr}_r \circ h)(\Gamma_i)={\rm pr}_r(e)=e.$

From this investigation, we get the following theorem.

\begin{thm}\label{col}
Let $\Sigma$ be a line arrangement and
$\pi_1({\mathbb{C}}^2-\Sigma) \simeq\mathop  (\bigoplus
_{i = 1}^n A_i ) \oplus \mathbb{Z}^l $
 where $A_i$ is a free group. Then for any multiple point $Q$ in the arrangement
 with $k$ lines, namely, $\{l_1,\ldots,l_k\}$, there exists $r,1\leq r\leq
 n$,and a projection onto $A_r$, $ \varphi_Q :G \rightarrow G$
 such that $A_r=\langle \varphi_Q(\Gamma_{1}),\dots,\varphi_Q(\Gamma_{k})\rangle \cong
 {\F}_{k-1}$. If $l_j$ is a line do not pass through the point, then
$\varphi_Q
  (\Gamma_j)=e$.

 Moreover, if ${p_1, \dots, p_m}$ are the multiple points of
 $ \Sigma$ and $n_i$ is the number of lines pass through the point $p_i$,
 then $G \cong \left( \bigoplus_{i = 1}^m C_i \right) \oplus B$,
 where $C_i \cong {\F}_{n_i-1}$. If $l$
 is a line which does not pass through $p_i$ and let $\Gamma$ be its corresponding
  generator, then ${ \rm pr}_i(\Gamma)=e$ ( where ${ \rm pr}_i$ is the projection onto $C_i$).
\end{thm}

\section{Main Theorems}
\begin{thm}\label{main theorem}  Let $\Sigma \subseteq {\C}^2$ be a line arrangement which has
no pair of parallel lines. Then if
$$\pi_1({\C}^2-\Sigma)=\mathop  \bigoplus _{i = 1}^r A_i  \oplus
\mathbb{Z}^l, $$ where  $A_i$ are free groups. Then
$\beta(\Sigma)=0$.

\end{thm}
\begin{proof} Assume by negation that $\beta(\Sigma)\neq 0$ which means
that there is at least one  cycle in the graph of $\Sigma$.
 Let us pick a minimal cycle in the following sense. The  cycle contains $r$  points,
 namely $\{p_1,\ldots,p_r\}$ and $p_i$ is connected in the cycle
only to $p_{i-1}$ and $p_{i+1}$ (the indices are taken modulo $r$).
In other words, a cycle with no sub-cycles.

  By Theorem \ref{col}, we can write $G = (\mathop  \oplus  _{i =
1}^r C_i ) \oplus B_1 $ where $B_1$ is not necessarily abelian.
 If $l$ is a line that does not pass through the points in the cycle
and  let $\Gamma$  be its corresponding generator,
 then $\ {\rm pr}_i(\Gamma)=e,1 \leq i \leq r$.
Define $N= \langle \Gamma_{X_1},...,\Gamma_{X_t}, \Gamma_1 \cdots \Gamma_n \rangle$, where
 $\Gamma_{X_1},...,\Gamma_{X_t} $  are the generators of lines that do not participate in an intersection
point  $p_i,\quad 1 \leq i \leq r$. Denote $Z=\Gamma_1\cdots
  \Gamma_n$. Let $H:=G/N$.

  Let $n_i$ be the number of lines passing through $p_i$.

On one hand, since we have $b:=\sum_{i=1}^r{(n_i-1)}$ lines
participating in the cycle, then if we denote by
$\Gamma_1,\ldots,\Gamma_b$ the generators associated with the
lines participating in the cycle, then $H\cong \langle
\Gamma_1,\ldots,\Gamma_b|\widetilde{R},\widetilde{Z}\rangle $
where $\widetilde{R}$ is the relations which are derived from the
original relations. Therefore, it is easy to see that
$\rank(\overline{H})\leq b-1$.

On the other hand, for  $ 1 \leq i \leq r \mbox{ and } 1 \leq j \leq t ,\quad { \rm pr}_i(\Gamma_{X_j})=e$
and also $\Gamma_1\cdots\Gamma_n\in Z(G) \Rightarrow
{ \rm pr}_i(\Gamma_1\cdots\Gamma_n)=e\Rightarrow { \rm pr}_i(N)=e $ and hence $C_i/N\cong C_i$.
\\  $\Rightarrow H=G/N\cong \left(\bigoplus_{i=1}^r{C_i\oplus
B_1}\right)/N.$
 Since $C_i\cong {\F}_{n_{i}-1}, \overline{C_i} \cong {\Z}^{n_i-1} $.

$
  \overline{H}=(\bigoplus_{i=1}^r{\overline{C_i})\oplus \overline{B_1/N}
 }\\
 \Rightarrow
 \overline{H}=(\bigoplus_{i=1}^r \mathbb{Z}^{n_i-1})\oplus \overline{B_1/N} \cong
 \mathbb{Z}^b
 \oplus \overline{B_1/N}
\\
 \Rightarrow \rank(\overline{H})\geq b,$ a
contradiction.
 \end{proof}

We are now ready to prove the converse of Fan's theorem.

\begin{thm} Let $\Sigma \subseteq \C \mathbb{P}^2$ be a line arrangement. If
$$\pi_1({\C}\mathbb{P}^2-\Sigma)=\mathop  \bigoplus _{i = 1}^r A_i  \oplus
\mathbb{Z}^l, $$ where  $A_i$ are free groups. Then
$\beta(\Sigma)=0$.
 \end{thm}
\begin{proof}
Suppose that $\Sigma \subset \Bbb C\mathbb{P}^2$ is an arrangement
of complex projective lines such that $\pi_1(\Bbb C\mathbb{P}^2
\setminus \Sigma)$ is isomorphic to a direct product of free
groups. Let $L_0\subset \Bbb C\mathbb{P}^2$ be a complex
projective line which is in general position to $\Sigma$. Then
$L_0$ and $\Sigma$ intersect at double points. Choose an arbitrary
complex projective line of $\Sigma$ and denote this line by $L_1$.
Note that $\Bbb C^2 \cong \Bbb C {\PP}^2\setminus L_0 \cong \Bbb
C{\PP}^2\setminus L_1$. By applying the product theorem of Oka and
Sakamoto \cite{OkSa}, we have
$$\pi_1(\Bbb C{\PP}^2 \setminus\left (\Sigma \cup L_0 \right )) = \pi_1(\Bbb (C{\PP}^2
\setminus L_1) \setminus ((\Sigma \setminus L_1) \cup L_0))$$
$$\cong \pi_1(\Bbb (C{\PP}^2
\setminus L_1) \setminus (\Sigma \setminus L_1))\oplus \pi_1(\Bbb
(CP^2 \setminus L_1) \setminus L_0) \cong \pi_1(\Bbb CP^2\setminus
\Sigma) \oplus \Bbb Z.$$ This calculation shows that $\pi_1(\Bbb
C{\PP}^2 \setminus (\Sigma \cup L_0))$ is a product of free group.
Note that
$$\pi_1(\Bbb C{\PP}^2 \setminus (\Sigma \cup L_0)) = \pi_1(\Bbb (CP^2
\setminus L_0) \setminus (\Sigma \setminus L_0)).$$ Note that
$\Sigma \setminus L_0$ is a union of complex lines in the complex
plane $\Bbb CP^2\setminus L_0$ such that no two lines of this
arrangement are parallel.

By Theorem \ref{main theorem}  $\beta (\Sigma \setminus L_0)=0$,
therefore  $\beta (\Sigma)=\beta (\Sigma \setminus L_0)=0$.

\end{proof}

\section{acknowledgements}

We wish to thank Dr. David Garber for the experience he shared
with us and the time he invested in improving the paper.  We are
extremely grateful for his grate efforts in order to make this
paper worthy.

We would also like to thank the anonymous referee for his fruitful
advices.

\end{document}